\documentclass[11pt]{article}
\usepackage{amsmath,amsthm,amssymb,amsfonts, mathrsfs, mathtools,
  enumitem, xcolor,bbm, tikz,tikz-cd}

\usepackage[ruled,vlined]{algorithm2e}
\usepackage{physics,hyperref, cleveref}
\usepackage[title]{appendix}

\usepackage[final]{microtype}
\usepackage{float}

\SetKwInput{KwInput}{Input}
\SetKwInput{KwOutput}{Output}

\usepackage[
    backend=biber,
    style=alphabetic,
    natbib=true,
    url=false, 
    doi=true,
    eprint=false,
    backref=true
    ]{biblatex}

\usepackage[right= 2.5cm, left=2.5cm, top= 2.5cm, bottom=1.9cm]{geometry}

\let\mbb\mathbb
\let\mc\mathcal

\let\mf\mathfrak
\let\mbf\mathbf

\newcommand{\1}{\mathbbm 1}
\newcommand{\N}{\mathbb{N}}
\newcommand{\Z}{\mathbb{Z}}
\newcommand{\Q}{\mathbb{Q}}
\newcommand{\R}{\mathbb{R}}
\newcommand{\h}{\mf h}
\newcommand{\C}{\mathbb{C}}

\renewcommand{\O}{\mc O}

\newcommand{\floor}[1]{\left\lfloor #1 \right\rfloor}

\newcommand*{\defeq}{\mathrel{\vcenter{\baselineskip0.5ex \lineskiplimit0pt
      \hbox{\scriptsize.}\hbox{\scriptsize.}}}%
  =}

\renewcommand\labelenumi{\textnormal{(\roman{enumi})}}
\renewcommand\theenumi\labelenumi

\renewcommand{\bar}{\overline}
\renewcommand{\tilde}{\widetilde}

\DeclareMathOperator{\cyc}{cyc}

\DeclareMathOperator{\sgn}{sgn}
\DeclareMathOperator{\Sl}{Sl}
\DeclareMathOperator{\SL}{SL}

\DeclareMathOperator{\Cl}{Cl}

\DeclareMathOperator{\Mat}{Mat}
\DeclareMathOperator{\Stab}{Stab}

\DeclareMathOperator{\Hom}{Hom}
\usepackage{calligra}

\DeclareMathOperator{\ord}{ord}
\DeclareMathOperator{\Gal}{Gal}

\theoremstyle{plain}
\newtheorem{thm}{Theorem}[section]
\newtheorem*{thm*}{Theorem}

\newtheorem{cor}[thm]{Corollary}

\newtheorem{lemma}[thm]{Lemma}
\newtheorem{prop}[thm]{Proposition}

\theoremstyle{definition}

\newtheorem{remark}[thm]{Remark}

\newtheorem{example}[thm]{Example}
\newtheorem{non-example}[thm]{Non-example}
\theoremstyle{remark}

\numberwithin{equation}{section}

\hypersetup{
  colorlinks,
  linkcolor={red!50!black},
  citecolor={purple!45!black},
  urlcolor={green!40!black}
}

\begin{document}
\title{Modular algorithms for Gross--Stark units and Stark--Heegner points}
\author{H\aa vard Damm-Johnsen\thanks{\href{mailto:havard.damm-johnsen@maths.ox.ac.uk}{havard.damm-johnsen@maths.ox.ac.uk}}
}

\date{}
\maketitle
\begin{abstract}
  In recent work, Darmon, Pozzi and Vonk explicitly construct a
  modular form whose spectral coefficients are $p$-adic logarithms of
  Gross--Stark units and Stark--Heegner points.  Here we describe how
  this construction gives rise to a practical algorithm for explicitly
  computing these logarithms to specified precision, and how to
  recover the exact values of the Gross--Stark units and Stark--Heegner
  points from them.

  Key tools are overconvergent modular forms, reduction theory of
  quadratic forms and Newton polygons.  As an application, we tabulate
  Brumer--Stark units in narrow Hilbert class fields of real quadratic
  fields with discriminants up to $10000$, for primes less than $20$,
  as well as Stark--Heegner points on elliptic curves. 
\end{abstract}
\tableofcontents
\section{Introduction}
The classical theory of complex multiplication, developed by
Kronecker, Weber, Fueter, Deuring, Shimura and others, gives an
explicit description of abelian extensions of imaginary quadratic
fields $K$. They are generated by \emph{elliptic units}, which are
canonical units in class fields of $K$. In \cite{stark1980}, Stark
proved that logarithms of elliptic units appear as special values of
derivatives of Hecke $L$-functions, and conjectured the existence of
units over arbitrary base fields, so-called \emph{Stark units}.  Heegner
and Birch used CM theory to construct points on modular curves, called
\emph{Heegner points}, also defined over class fields. By mapping
these to elliptic curves, Gross and Zagier \cite{gross1986}
made important contributions towards the BSD conjecture.
 
Let $F$ be a real quadratic field and $p$ a rational prime. While
there is no direct analogue of the construction of elliptic units over
$F$, Gross \cite{gross1981} constructed what is now known as
\emph{Gross--Stark units}, formal powers of units in class fields of
$F$, and formulated an analogue of Stark's conjectures for these. His
conjecture related special values of derivatives of $p$-adic
$L$-functions to local norms of Gross--Stark units, and was proved in
\cite{dasgupta2011}. Recent work of Dasgupta and Kakde
\cite{dasgupta2023} on the Brumer--Stark conjecture refines this by
removing the norm. The computation of Gross--Stark units over quadratic
fields has been studied in \cite{tangedal2013} when $p$ splits in $F$,
and \cite{dasgupta2007a}, \cite{dasgupta2021}, and
\cite{fleischer2022} for $p$ inert.

By analogy with Heegner points, Darmon's work \cite{darmon2001} uses
$p$-adic analysis to construct points on elliptic curves. These
so-called \emph{Stark--Heegner points} are conjectured to be defined
over ring class fields of $F$. While this conjecture is still wide
open in general, it is supported by extensive computations. Algorithms
for computing Stark--Heegner points are given in \cite{darmon2006b},
\cite{guitart2015} and \cite{darmon2021a}.

In the last reference cited, Darmon and Vonk use the theory of rigid
analytic cocycles to provide a common framework for Gross--Stark units
and Stark--Heegner points. Their framework also gives an analogue of
singular moduli for real quadratic fields, for which the techniques in
this paper are expected to generalise. Rigid analytic cocycles are
used in subsequent work of Darmon, Pozzi and Vonk \cite{darmon2021b}
to show that diagonal restrictions of certain $p$-adic families of
Hilbert modular forms gives an explicitly computable modular form
whose spectral expansion contains both Gross--Stark units and
Stark--Heegner points.

More specifically, the authors construct a classical modular form $G$
from a parallel weight $1$ Hilbert Eisenstein series $E_{1,1}$ over
$F$ as follows: first, they define the \emph{anti-parallel weight
  deformation} of $E_{1,1}$, modify by a linear combination of
Eisenstein families, restrict the argument to the diagonally embedded
$\mf h$ in $\mf h \times \h$, and take the first order derivative of the
family. This is shown to be a $p$-adic modular form, to which they
finally apply Hida's ordinary projector to get the modular form $G$.
A consequence of the main theorems in \cite{darmon2021} and
\cite{darmon2021b} along with a conjecture in \cite[\S 3]{darmon2020}
which relates values of cocycles attached to cuspidal eigenforms to
Stark--Heegner points, is the following:

\begin{thm}\label{thm:vague-theorem}
  The form $G \in M_{2}(\Gamma_{0}(p))$ is non-zero if $F =
  \Q(\sqrt D)$ has no unit of negative norm, and satisfies
  \begin{equation*}
    \left\langle G,f \right\rangle_{\Gamma_{0}(p)}  =
    \begin{cases}\frac{1}{p-1}\log_{p} u &\text{ if } f =
      E_{2}^{(p)}, \\
      L_{\mathrm{alg}}(1,f)\log_{E_{f}} P_{f} &\text{ if } f \text{ is a cuspidal eigenform with
      rational coefficients}.
    \end{cases}
  \end{equation*}
\end{thm}
Here $u$ is a Gross--Stark unit, $E_{2}^{(p)}$ the Eisenstein series on
$M_{2}(\Gamma_{0}(p))$, $L_{\mathrm{rat}}(1,f)$ the rational part of the
special value $L(1,f)$ of the $L$-function attached to $f$, $E_{f}$
the elliptic curve associated to $f$ via the Eichler--Shimura
construction, $\log_{E_{f}}$ the formal logarithm on $E_{f}$, and $P_{f}$
a Stark--Heegner point on $E_{f}$, conjecturally defined over the
narrow Hilbert class field of $F$.

A more precise statement is found in \cref{thm:eigenspansion}.

The goal of this paper is to show that the steps defining $G$ can be
made completely explicit in a computer algebra system such as
\texttt{sage} \cite{thesagedevelopers2022} or \texttt{magma}
\cite{bosma1997}, and in particular we can compute the spectral
coefficients of $G$ to arbitrary precision.  A key tool is the
algorithms for overconvergent modular forms due to Lauder
\cite{lauder2011,lauder2014}, with necessary modifications for
$p \in \{2,3\}$ from \cite{vonk2015}. As a proof of concept, we compute
tables of
\begin{itemize}[itemsep=-4pt]
\item Gross--Stark units over $\Q(\sqrt{D})$ for fundamental
  discriminants $D < 10000$ and $p< 20$, and
\item Stark--Heegner points on elliptic curves for $D < 100$,
  $p<20$. This can be viewed as a numerical verification of the
  conjecture in \cite{darmon2020}.
\end{itemize}
For $p$ equal to $2$ or $3$, these tables are virtually complete, with
only a handful of omissions due to the large height of the polynomials. 
\begin{example}
  Let $D = 8441 = 23 \cdot 367$. Then $F\defeq \Q(\sqrt D)$ has narrow
  class number $26$, and combining \cref{alg:GS-unit} and
  \cref{alg:rec-algdep} gives the polynomial
\begin{equation}
  \begin{alignedat}[c]{2}
  3^{43}x^{26} - &3^{28} \cdot 74700593x^{25}  &&+3^{21} \cdot
  413213377697x^{24}\\
  - &3^{14} \cdot 1491793680346193x^{23}  &&+3^{11} \cdot 48103058975883121x^{22}\\
  - &3^{8} \cdot 1176950719953501830x^{21}  &&+3^{8} \cdot
  841442767734656470x^{20}\\
  - &3^{6} \cdot 5230173358710191479x^{19}  &&+3^{7} \cdot
  1983729129037937219x^{18}\\
  - &3^{5} \cdot 28800297384178354201x^{17}  &&+3^{6} \cdot
  13798304822142405250x^{16}\\
  - &3^{2} \cdot 1314012089988186633625x^{15}  &&+3^{2} \cdot
  1350085297035065778356x^{14}\\
  - &12074610496660929030725x^{13}  &&+3^{2} \cdot
  1350085297035065778356x^{12}\\
  - &3^{2} \cdot 1314012089988186633625x^{11}  &&+3^{6} \cdot
  13798304822142405250x^{10}\\
  - &3^{5} \cdot 28800297384178354201x^{9}  &&+3^{7} \cdot
  1983729129037937219x^{8}\\
  - &3^{6} \cdot 5230173358710191479x^{7}  &&+3^{8} \cdot
  841442767734656470x^{6}\\
  - &3^{8} \cdot 1176950719953501830x^{5}  &&+3^{11} \cdot
  48103058975883121x^{4}\\
  - &3^{14} \cdot 1491793680346193x^{3}  &&+3^{21} \cdot 413213377697x^{2}\\
  - &3^{28} \cdot 74700593x  &&+3^{43}.&&&
  \end{alignedat}
\end{equation}
The roots of this polynomial are $3$-units generating the narrow Hilbert class
field of $F$, a degree $52$ extension of $\Q$, and their square roots
are Gross--Stark units attached to ideal classes in $F$, as defined in \cref{sec:gs-units}.
\end{example}

\begin{example}
  Let $p=11$ and consider $E: y^2 + y = x^3 - x^2 - 10x - 20$, a
  model for $X_{0}(11)$. Using \cref{alg:stark-heegner} we find the
  following points on $E$: 
  \begin{table}[H]
    \centering
    \begin{tabular}{c|c|c}
      $D$ & $X$ & $Y$ \\ \hline
      $21$ & $x^{2} + 3 x + 4$ & $x^{2} + 3 x + 4$ \\
      $24$ & $x^{2} + 8$ & $x^{2} + 10 x + 57$ \\
      $28$ & $x^{2} + \frac{71}{16} x + \frac{23}{4}$ & $x^{2} - \frac{101}{64} x + \frac{599}{64}$ \\
      $57$ & $x + \frac{1065}{304}$ & $x^{2} + x + \frac{1130412905}{28094464}$ \\
      $76$ & $x + \frac{1065}{304}$ & $x^{2} + x + \frac{1130412905}{28094464}$ \\
    \end{tabular}
  \caption{Table of Stark--Heegner points on
      $E: y^2 + xy + y = x^3 - x^2 - x - 14$, for $D<100$}  
  \end{table}
  For each $D$, the points on $E$ whose $x$ and $y$ coordinates are
  given by roots of $X$ and $Y$ respectively, are defined over the
  narrow Hilbert class field of $\Q(\sqrt D)$. 
  
\end{example}

The paper is structured as follows: in \cref{sec:general} we first
give a precise definition of Gross--Stark units and Stark--Heegner
points, then discuss the results of \cite{darmon2021b} and explain how to
use the classical reduction theory of indefinite quadratic forms to
greatly improve the efficiency of the resulting algorithms.  Next, in
\cref{sec:gs-units} we explain how to use the Brumer--Stark conjecture
to recover a Gross--Stark units from its $p$-adic logarithm, and how to
compute a Stark--Heegner point from its formal logarithm. We also
discuss how to verify the correctness of the data computed.  Finally,
we present some data and make some observations.

The paper is supplemented by the implementations of the algorithms in
a \texttt{magma} library: \url{https://github.com/havarddj/drd}. There
is also an implementation in
\texttt{sage}, \url{https://github.com/havarddj/hilbert-eisenstein}
which is work in progress.

\textbf{Acknowledgements:} I am very grateful to Jan Vonk for
suggesting the problem and for continued guidance and great
suggestions, and to James Newton for many helpful conversations and
comments on the paper. Thanks to Alex Braat for suggesting the
statement of \cref{lemma:central-ext}, Samuel Frengley for help with
\texttt{magma}, and to Alex Horawa and George Robinson for enlightening
conversations.

\section{The modular algorithm}
\label{sec:general}
\subsection{Notation}
For the remainder of the paper, $F$ will denote a real quadratic
extension of $\Q$ of discriminant $D$, and $\O_{F}$ its ring of
integers. If $\alpha \in F$, then $\alpha'$ denotes its conjugate. 

We let $\Cl^{+}$ be the narrow Hilbert class group of $F$, so that
$\Cl^{+} \cong G \defeq \Gal(H/F)$ where $H$ is the narrow Hilbert class
field of $F$, the maximal abelian extension of $F$ unramified at all
finite places, of degree $h^{+}$. For $\sigma \in G$, the corresponding class
in $\Cl^{+}$ is denoted $A_{\sigma}$, and conversely a class $A$ determines
an automorphism $\sigma_{A} \in G$.  The narrow class group is strictly
larger than the wide class group if and only if $F$ has no units of
norm $-1$, and in light of \cref{thm:vague-theorem} we restrict our
attention to this case. Then $H$ is a CM extension of the (wide)
Hilbert class field, and multiplication by complex conjugation
corresponds to the involution $A \mapsto A[\sqrt D]$ on the class group. Let
$p$ be a rational prime inert in $F$. Then $p$ splits completely in
$H$, and we fix a prime $\mf P$ of $H$ above $p$. This determines an
isomorphism of completions $F_{p} \cong H_{\mf P}$. A function
$f \colon \Cl^{+} \to \C$ is odd if $f(A[\sqrt D]) = -f(A)$ for all
$A \in \Cl^{+}$. The field generated by the values of a character
$\psi$ is denoted by $\Q(\psi)$.

We use fraktur letters to denote ideals, and in particular $\mf P$
will always denote a non-zero prime ideal in $H$ dividing $p$. Moreover, the
symbol $\mf d$ will be reserved for the different ideal of $F$. We say
an element $\alpha \in F$ is totally positive if $\rho(\alpha) > 0$ for all embeddings
$\rho \colon F \hookrightarrow \R$, and we write $\alpha \gg 0$. If $X \subset F$ is any subset, set
$X_{+} \defeq \{ \alpha \in X : \alpha \gg 0\}$.

Given an integral ideal $\mf a$ of $F$, let
$N(\mf a) \defeq \#(\O_{F}/\mf a)$, and this extends to fractional
ideals by $N(\mf a/\mf b) \defeq N(\mf a)/N(\mf b)$, and to elements
$\alpha \in F^{\times}$ by $N(\alpha) = N((\alpha))$, where $(\alpha)$ denotes the fractional
ideal generated by $\alpha$. By convention, we also set $N(x) = x^{2}$ when
$x$ is an indeterminate. If $\mf p$ is a non-zero prime ideal and
$\alpha \in F^{\times}$, then we set
$|\alpha|_{\mf p} = N(\mf p)^{-\ord_{\mf p}\alpha}$, where
$\ord_{\mf p} \alpha$ denotes the power of $\mf p$ appearing in the prime ideal
factorisation of $(\alpha)$.  For any number field $K$, $\mu(K)$ denotes the
set of all roots of unity in $K$. We also define the $p$-units in $K$
as the group
$\O_{K}[1/p]^{\times} \defeq \{\alpha \in K : \ord_{\mf q}(\alpha) = 0 \text{ if } \mf
q \nmid p \}$. Equivalently, this is the $S$-units of $K$ where $S$
consists of the places of $K$ above $p$. This is a finitely generated
abelian group by a version of Dedekind's unit theorem,
\cite[Thm.~12.27]{harari2020}.

\subsection{Gross-Stark units and Stark--Heegner points}
\label{sec:GS-SH}

\Cite[Prop.~3.8]{gross1981} proves the existence and uniqueness of a
``formal power of a $p$-unit'' $u \in \O_{H}[1/p]^{\times}\otimes \Q$ characterised
by the properties
\begin{equation}
  \label{eq:17}
  \ord_{\mf P} \sigma(u) = L(0,A_{\sigma}) \text{ for all } \sigma \in G
  \qq{and} \bar u = 1/u,
\end{equation}
where the bar denotes complex conjugation, and $L(s,A_{\sigma})$ is the
partial $L$-function defined by the Dirichlet series
$L(s,A_{\sigma}) = \sum_{\mf a \le \O_{F}, \, [\mf a] = A_{\sigma}} N(\mf a)^{-s}$,
which admits a meromorphic continuation to $\C$ in the usual manner.
This depends only on the choice of prime $\mf P$ of $H$ above $p$.  In
\cite[Eq.~(4)]{darmon2021b}, the authors twist by elements of $G$ to
get units $u_{A} \defeq \sigma_{A}(\bar u)$ indexed by $A \in \Cl^{+}$, equal to
$u_{\tau}$ when $A = [\Z +\tau Z]$ in their
notation. It is therefore characterised by
\begin{equation}
  \label{eq:5}
  \ord_{\mf P^{\sigma}} u_{A} = -L(0,AA_{\sigma^{-1}}) \text{ for all } \sigma \in G
  \qq{and} \bar u_{A} = 1/u_{A}.
\end{equation}
This is referred to as the \emph{Gross--Stark unit attached to
  $A$}. Note that these are all $G$-conjugate: $\sigma (u_{A}) = u_{AA_{\sigma}}$.

The Brumer--Stark conjecture, proven up to powers of $2$ in
\cite{dasgupta2023}, implies that $u_{A}^{e}$, where $e = \#\mu(H)$,
gives an element of $\O_{H}[1/p]^{\times}$. More precisely, there exists an
element $\epsilon \in \O_{H}[1/p]^{\times}$ such that
$\epsilon \otimes 1 = e \cdot u$ and such $H(\sqrt[e]{\epsilon})/F$ is an abelian extension. We set
$\epsilon_{A} \defeq \bar \sigma_{A}(\bar \epsilon)$, which we refer to as the \emph{Brumer--Stark
unit attached to $A$}.  These are the units we compute in
\cref{sec:gs-units}. An immediate consequence of the second part of
\cref{eq:5} is that $\epsilon_{A}$ always lies on the unit circle under any
embedding $H \hookrightarrow \C$. For the remainder of the paper, we will assume the
full Brumer--Stark conjecture. Our computations can then be viewed as a
verification of the conjecture.

We also attach a Gross--Stark unit to a character
$\psi :G \to \C^{\times}$ by setting
\begin{equation}
  \label{eq:9}
  u_{\psi} \defeq \prod_{A \in \Cl^{+}}u_{A}^{\psi(A)} = \prod_{\sigma \in G} \sigma(\bar u)^{\psi(A_{\sigma})},
\end{equation}
which lies in $\O_{H}[1/p]^{\times}\otimes \Q(\psi)$, and satisfies
$\ord_{\mf P u_{\psi}} = - L(0,\psi)$ and
$\sigma(u_{\psi}) = \bar \psi(A_{\sigma}) u_{\psi}$ for all
$\sigma \in G$. This is compatible with the notation in
\cite{dasgupta2011}.\footnote{However, it is different from the
  formula in {\cite[Eq.~51]{darmon2021b}}, in which $u_{\psi}$ depends on
  $\tau$, and the corresponding formula for $\ord_{\mf P} u_{\psi}$ in the
  proof of Lemma 3.5 is off by a factor of $\psi(\sigma_{A})$, or
  $\psi(\tau)$ in their notation.}

Stark--Heegner points $P_{f,\psi}$ are defined in \cite{darmon2001} and
\cite{dasgupta2005}, and for brevity we give a description of their
properties instead of a strict definition. They are defined on the
modular Jacobian $J_{0}(p)$, which coincides with an isogeny class of
elliptic curves when the genus of $X_{0}(p)$ is one. More generally,
if $J_{0}(p)$ splits into a product of abelian varieties of which one
is an elliptic curve $E$, then there exists a cuspidal eigenform $f \in
S_{2}(\Gamma_{0}(p))$  such that $E$ is isogenous to $E_{f}$, and $P_{\psi,f}$
gives a point on these.

Pick an elliptic curve $E_{f}$ in the isogeny class. In this setting,
$P_{f,\psi}$ comes from an element of $F_{p}$ defined via $p$-adic
analytic methods. By \cite[Thm.~14.1]{silverman2009}, $E_{f}(F_{p})$
is isomorphic to $F_{p}^{\times}/q^{\Z}$ where $q$ is the Tate parameter
attached to $E_{f}$. We can find an explicit isomorphism
$E_{f}(F_{p}) \to F_{p}^{\times}/q^{\Z}$ by first finding an isomorphism
between $E_{f}$ and the corresponding Tate curve $E_{q}$ by computing
their Weierstra\ss\ equations and using the intrinsic
\texttt{IsIsomorphic} in \texttt{magma}, and then computing the
isomorphism $E_{q}\to F_{p}^{\times}/q^{\Z}$ using the formulae in \cite[\S
C.14]{silverman2009}. 

This gives a point $P_{\psi,f}$ in $E_{f}(F_{p})$. However, it is
conjectured in \cite{darmon2001} that it is actually defined over $H$
via the embedding $H \hookrightarrow H_{\mf P} \cong F_{p}$, and in \cref{sec:SH} we
verify this computationally.

\subsection{Diagonal restriction derivatives}
Let $\psi$ be an odd character on $\Cl^{+}$. Following
\cite{darmon2021b} we consider the Hilbert modular Eisenstein series
$E_{1,1}(\psi)$ of parallel weight $1$ whose $q$-expansion at the cusp
$\mf d$ is given by the series 
\begin{equation}
  \label{eq:6}
  E_{1,1}(\psi)_{\mf d} = \sum_{\nu \in \mf d^{-1}_{+}} \sigma_{0,\psi}(\nu
  \mf d)q^{\tr \nu},
\end{equation}
where $\sigma_{0,\psi}(\nu \mf d)$ is the divisor sum
\begin{equation}
  \label{eq:4}
  \sigma_{0,\psi}(\nu \mf d) \defeq \sum_{\mf a \mid \nu \mf d} \psi(\mf a).
\end{equation}

For $p$ a rational prime inert in $F$, we also define the
$p$-stabilisation of $E_{1,1}(\psi)$ by
$E^{(p)}_{1,1}(\psi)(z_{1},z_{2}) \defeq E_{1,1}(\psi)(z_{1},z_{2}) -
pE_{1,1}(\psi)(pz_{1},pz_{2})$. There is a certain $p$-adic family of
modular forms $\mc F^{+}$, a linear combination of two Eisenstein
families along with the \emph{anti-parallel weight deformation},
whose weight $1$ specialisation equals $E_{1,1}^{(p)}(\psi)$.  Note that
$\mc F^{+}$ is different from the parallel weight Eisenstein family used in
\cite{darmon2021}, and computing its $q$-expansion requires a fairly
complicated argument using Galois deformation theory, the details of
which are in \cite[\S 3]{darmon2021b}. Since
$E_{1,1}^{(p)}(\psi) (z,z)$ is a classical modular form of level
$1$ and weight $2$ and therefore identically $0$,
$E_{1,1}^{(p)}(\psi)$ vanishes along the diagonally embedded copy of
$\mf h $ in its domain $\mf h\times \mf h$.  Taking the derivative of
$\mc F^{+}$ in the weight space and restricting to weight $1$ then
gives an overconvergent modular form in one variable, denoted
$\partial f^{+}_{\psi}$. We refer to this as the \emph{diagonal restriction
  derivative}, and its $q$-expansion is given as follows:

\begin{prop}[{\cite[Prop.~4.6]{darmon2021b}}]\label{prop:OCdrd}
  The diagonal restriction derivative is an overconvergent modular
  form of weight $2$ and tame level $1$
  \begin{equation}
    \label{eq:25}
    \partial f^{+}_{\psi}(q) = \frac{1}{2}\log_{p}(u_{\psi}) - \sum_{n = 1}^{\infty}
    \sum_{\substack{\nu \in \mf d_{+}^{-1} \\ \tr \nu = n} }\sum_{\substack{\mf a \mid (\nu)\mf
      d \\  (\mf a ,p) =1}} \psi(\mf a) \log_{p}\qty(\frac{\nu \sqrt{D}}{N(\mf a)})q^{n},
  \end{equation}
  with rate of overconvergence $r$ for each $r < p/(p+1)$. 
\end{prop}
 The symbol $\log_{p}$ denotes the $p$-adic
logarithm, defined by the power series $\log_{p}(1-x) =
\sum_{n=1}^{\infty}x^{n}/n$ on its domain of convergence in $\O_{F_{p}}$, and
extended by setting $\log_{p}(p) = \log_{p}(\zeta) = 0$ for any root of
unity $\zeta$ in $F_{p}$. To evaluate this at elements of $F$, we identify
$F$ with its image in $F_{p}$. 

Applying Hida's ordinary projection operator $e_{\ord}$ to
$\partial f^{+}_{\chi}$ gives a classical modular form of level $\Gamma_{0}(p)$ and weight
$2$. The space of such forms is spanned by the Eisenstein series
\begin{equation}
  \label{eq:73}
  E_{2}^{(p)}(z) =\frac{p-1}{24} + \sum_{n=1}^{\infty} \Big(\sum_{\substack{d | n
      \\ (d,p) = 1}} d\Big)q^{n},
\end{equation}
along with eigenforms $f$, which we normalise so that
$a_{1}(f) = 1$ in the $q$-expansion at $\infty$. 

\begin{thm}\label{thm:eigenspansion}
  Set $F = \Q(\sqrt{D})$ and let $p$ be a prime inert in $F$.  Assume
  conjecture 3.19 in \cite[\S 3]{darmon2020}, and write
\begin{equation}
  \label{eq:69}
  e_{\ord}(\partial f^{+}_{\psi}) = \lambda_{0}E_{2}^{(p)} + \sum_{f} \lambda_{f}f, \qq{where} \lambda_{0},\lambda_{f} \in F_{p}.
\end{equation}
Then $\lambda_{0} = \frac{1}{p-1}\log u_{\psi}$, and if $a_{n}(f) \in \Q$ for all $n$, then
$\lambda_{f} = L_{\mathrm{alg}}(1,f) \log_{E_{f}}(P_{\psi,f})$, where
$P_{\psi,f}$ is a Stark--Heegner point on $E_{f}$, the elliptic
curve attached to $f$ by the Eichler--Shimura
construction, and $L_{\mathrm{alg}}(1,f)$ is the algebraic part of the
value $L(1,f)$.
\end{thm}

\begin{proof}[Proof of \cref{thm:eigenspansion}]  
  By \cite[Prop.~4.7]{darmon2021b},
  $G\defeq e_{\ord}(\partial f^{+}_{\psi})$ can be written as a generating
  series\footnote{There is a sign missing in the proof of Thm.~4.8
    which propagates back to Prop.~4.7. As written, the constant term of the
    Eisenstein series in the spectral expansion is off by a factor of
    $-1$.  We assume here that the statement of Thm.~4.8 is
    correct as written. We anticipate that this will be clarified in
    the published version of \cite{darmon2021b}.}
  \begin{equation}
    \label{eq:74}
   2G(z) = \log_{p}(u_{\psi}) + \sum_{n=1}^{\infty} \log_{p}(T_{n}J_{w}[\psi])q^{n}.
  \end{equation}
  Meanwhile, by \cite[eq.~29]{darmon2021b} the cocycle $J_{w}$
  decomposes as follows:
  \begin{equation}
    \label{eq:75}
    J_{w} = \frac{2}{p-1} J_{\mathrm{DR}} + 2\sum_{f}L_{\mathrm{alg}}(1,f)
    J_{f}^{-} \bmod{J^{\Z}_{\mathrm{univ}}},
  \end{equation}
  Plugging the expression for $J_{w}$ into the $n$-th Fourier
  coefficient for $n \ge 1$ coprime to $p$, we obtain 
  \begin{subequations}
    \begin{align}
      a_{n}(G) &= \frac{2}{p-1}\log_{p} T_{n}J_{\mathrm{DR}}[\psi] + 2\sum_{f}
                  L_{\mathrm{alg}}(1,f) \log_{p}T_{n}J_{f}^{-}[\psi] \\
                &= \frac{2}{p-1}\log_{p}( J_{\mathrm{DR}}[\psi])\cdot a_{n}(E_{2}^{(p)}) + \sum_{f}
                  L_{\mathrm{alg}}(1,f) \log_{p}(J_{f}^{-}[\psi])\cdot a_{n}(f).
    \end{align}
  \end{subequations}
  Theorem B of \cite{darmon2021b} combined with the proof of Theorem
  4.8 in the same paper implies that
  $J_{\mathrm{DR}}[\psi] = u_{\psi}^{24}$, and conjecture 3.19 in
  \cite{darmon2020} that $J_{f}^{-}[\psi]$ maps to
  $P_{\psi,f} \in E_{f}(F_{p})$ under the Tate uniformisation. Denoting the
  composite of the Tate map and $\log_{p}$ by $\log_{E_{f}}$, we get that
  \begin{equation}
    \label{eq:8}
    a_{n}(G) = \frac{24}{p-1}\log_{p} (u_{\psi})\cdot a_{n}(E_{2}^{(p)}) + \sum_{f}L_{\mathrm{alg}}(1,f)\log_{E_{f}}
    P_{\psi,f}\cdot a_{n}(f)
  \end{equation}
  As in the proof of \cite[Prop.~4.7]{darmon2021b}, there exists a
  modular form in $M_{2}(\Gamma_{0}(p))$ with prime to $p$ coefficients
  $a_{n}(G)$, which we denote by $g$. Now $g-G $ is an oldform in
  $M_{2}(\Gamma_{0}(p))$ as all its coefficients of index coprime to
  $p$ vanish, hence equals $0$, and this completes the proof.
\end{proof}

This construction can be made completely explicit in a computer
algebra system such as \texttt{magma} or \texttt{sage}, at least to
finite $p$-adic precision:

\begin{enumerate}
\item Compute the terms $\{a_{n}\}_{n =1}^{M}$ of the $q$-expansion of
  $\partial f^{+}_{\psi}$ in \cref{eq:25} up to a certain bound $M$ by enumerating
  the elements $\nu \in \mf d^{-1}_{+}$ of trace $n$ and factorising
  $\nu \mf d$. Since $\log_{p}(xy) = \log_{p}x + \log_{p} y$ for any
  $x,y \in F_{p}$, we only need to evalute this once per $n$. 
\item Compute a basis for the space of overconvergent modular forms to
  sufficiently high precision using \cite[Algorithm 1]{lauder2011}.
\item Solve for $\partial f^{+}_{\psi}$ and its constant term in this basis.
\item Compute the ordinary projection as a matrix on the basis, and
  apply to the vector defining $\partial f^{+}_{\psi}$ to get
  $e_{\ord} (\partial f^{+}_{\psi})$. This is described in detail in step (6) of
  \cite[Alg.~2.1]{lauder2014}.
\item Solve for $e_{\ord} (\partial f^{+}_{\psi})$ in an eigenbasis of
  $M_{2}(\Gamma_{0}(p))$, which can be found explicitly using built-in
  methods in \texttt{sage} and \emph{magma}.
\end{enumerate}

In practice, the first step is very slow due to the cost of evaluating
$\psi(\mf a)$ for many $\mf a$. Moreover, the coefficients of $\partial
f^{+}_{\psi}$ lie in an extension of $F_{p}$ generated by the values of
$\psi$, which is typically large if the narrow class number of $F$ is.

\subsection{Improvements using quadratic forms}
To get around these difficulties, we combine two observations: the
first is that if we split the sum into a sum over classes
$A \in \Cl^{+}$, then it suffices to compute sums corresponding to all
pairs $(\nu, \mf a)$ where $\mf a \mid \nu \mf d$ and $\mf a$ has class
$A$ in the narrow class group, which lie in $F_{p}$. The second is
that by the correspondence between ideals of $\Q(\sqrt D)$ and
indefinite quadratic forms of discriminant $D$, we can use reduction
theory to enumerate all such ideals.

\begin{prop}[{\cite[Ex.~7.21]{cox2011}}]\label{prop:QF-Id-corr}
  There is a bijection between ideals of $\Q(\sqrt D)$ and indefinite
  quadratic forms of discriminant $D$, given by
  \begin{subequations}
    \begin{align}
      \mf a  = \alpha \Z + \beta \Z
      &\mapsto \frac{N(x\alpha - y\beta)}{N(\mf a)}, \\
      Q(x,y) =  ax^{2}+bxy + cy^{2}
      &\mapsto \begin{cases}
           a\Z + a\tau \Z &\text{ when } a > 0,\\
           \sqrt D (a\Z + a\tau \Z) &\text{ when } a < 0.\\
         \end{cases}
    \end{align}
  \end{subequations}
  Here $\tau$ is the root of $Q(x,1)$ satisfying $\tau > \tau'$.

  This bijection respects the class group structure:
  two ideals are equivalent if and only if the corresponding quadratic
  forms are equivalent under the action of $\SL_{2}(\Z)$,
  \begin{equation}
    \label{eq:1}
    \mqty(r  & s \\ t & u) \cdot Q = Q(rx + sy, tx + uy).
  \end{equation}
\end{prop}
We say that an indefinite quadratic form $Q(x,y) = ax^{2}+bxy+cy^{2}$
is \textbf{reduced} if $|\sqrt D - 2|a|| < b < \sqrt D$. Any given form
is equivalent to finitely many reduced forms.

\begin{prop}\label{prop:QF-ideal-bij}
  Let $F = \Q(\sqrt{D})$ be a real quadratic field and
  $A \in \Cl^{+}$ a fixed class with associated reduced
  quadratic form $Q_{0}$. Then there is a bijection between
  \begin{equation}
    \label{eq:99}
    \mbb I(n,A) \defeq \qty{(\mf a, \nu) :\hspace{3pt} \parbox[c]{3cm}{\raggedright $ \nu \in
        \mf d^{-1}_{+}$, $\tr \nu= n$ \\ $\mf a \mid (\nu)\mf
        d,\ [\mf a] = A$}}
  \end{equation}
and 
  \begin{equation}
    \label{eq:88}
    M(n,A) \defeq \qty{(Q = ax^{2}+bxy+cy^{2},\gamma) 
      :\hspace{3pt} \parbox[c]{3cm}{\raggedright $\gamma \in N_{n}, \ Q \sim Q_{0}^{\gamma}$, \\
        $ a > 0 > c$ }},
  \end{equation}
  where $N_{n}$ is a set of double coset representatives of
  \begin{equation}
    \label{eq:699}
    \Sl_{2}(\Z) \setminus \{\gamma \in \Mat_{2}(\Z) : \det \gamma = n\} / \Stab_{\Sl_{2}(\Z)}(Q_{0}).
  \end{equation}
\end{prop}
\begin{proof}
  This is essentially \cite[Lemma 4.1]{lauder2022}, except we
  identify $\tau$ with its associated quadratic form.
\end{proof}
We call an element $Q \in M(n,A)$ a \emph{nearly reduced form} since although it might not
be reduced in the strict sense, it is an element of the
reduced cycle of $Q_{0}$, see \cite[Ch.~6]{buchmann2007}. Note that
$N_{n}$ can be found as a subset of the coset representatives of
$\Sl_{2}(\Z) \setminus \{\det \gamma = n\}$, which we can choose to be 
\begin{equation}
    \mqty(n/m & j \\ 0 & m), \qq{} m | n, \ 0 \le j \le m-1,\  (m,n/m) = 1.
\end{equation}

The sets $M(n,A)$ and $M(d, A)$ for $d \mid n$ are not independent: if
$Q \sim Q_{0}^{\gamma_{n}}$ for some $\gamma_{n} \in N_{n}$, then we can find
corresponding elements $\gamma_{d}$ and $\gamma_{n/d}$ such that
$\gamma_{n} = \gamma_{n}\gamma_{n/d}$, and so we can generate it in
$M(n,A)$ by applying suitable Hecke matrices to pairs in $M(d,A)$.

This gives the following recursive algorithm for computing $M(n,A)$:

\begin{algorithm}[H]
\caption{Compute set the $M(n,A)$ of nearly reduced forms}
\DontPrintSemicolon
\KwInput{
  \begin{itemize}[itemsep=-4pt]
  \item A fundamental discriminant $D$,
  \item A class $A$ in $\Cl^{+}$ represented by a quadratic form $Q_{0}$,
  \item A positive integer $n$. 
  \end{itemize}
}
\KwOutput{
  A set of sets $\{M(d,A)\}$ indexed by divisors $d \mid n$.
}

\If{ $n = 1$}{\Return $\{\{Q,\mbf{1}\}\}$}
$M_{n} \gets \emptyset$ \tcp*{Initialise $M_{n}$}
$p \gets$ smallest prime dividing $n$ \;
$d \gets n/p$ \;
$M_{d} \gets M(d,A)$ \; 
$ \displaystyle   H_{p}\gets \qty{\mqty(p/m & j \\ 0 & m) :  m \in \{1,p\}, \ 0
  \le j \le m-1}$ \;
  \For{$(Q_{d},\gamma_{d}) \in M_{d}$}{
    \For{$\delta\in H_{p}$}{
    $Q' \gets Q_{d}^{\delta}$\;
    \If{$Q' \not \sim_{\SL_{2}(\Z)} Q$ for all $(Q,\gamma) \in M_{n}$}{
      $Q_{1},\ldots, Q_{c} \gets \mathtt{ReducedCycle}(Q')$ \;
      $M_{n} \gets M_{n} \cup \{(Q_{1},\delta\gamma_{m})\ldots, (Q_{c},\delta\gamma_{m})\}$
    }
}}
\Return{ $\{M_{d} : d \mid n\}$\label{alg:NR-forms}}
\end{algorithm}

\begin{remark}\label{rmk:QF-invo}
  Note that computing $M(n,A)$ gives $M(n, A[\sqrt D])$ for free using
  the involution $ax^{2}+bxy + cy^{2}\mapsto -cx^{2}-bxy-ay^{2}$.
\end{remark}

It is convenient to work with so-called \emph{odd indicator functions
  on $\Cl^{+}$}, meaning functions of the form
\begin{equation}
  \label{eq:2}
  \1^{*}_{A}(B) \defeq \1_{A}(B) - \1_{A[\sqrt D]}(B) =
  \begin{cases}
    1 &\qq{if} B =A, \\
    -1 &\qq{if} B = A[\sqrt D], \\
    0 &\qq{otherwise}. 
  \end{cases}
\end{equation}
We can pass between odd characters and odd indicator functions via the
change of basis formulae
\begin{equation}
  \label{eq:3}
  \psi(A) = \frac{1}{2} \sum_{B \in \Cl^{+}}\psi(B)\1^{*}_{B}(A) \qq{and}
  \1^{*}_{A}(B) = \frac{2}{h^{+}}\sum_{\psi \text{ odd}} \psi(B)\bar \psi(A).
\end{equation}

These are simple consequences of the orthogonality relations for
characters, see \cite[\S 2.3]{serre1977a}. By linearity, we obtain the
following version of \cref{prop:OCdrd}:

\begin{cor}
  Fix an indefinite quadratic form $Q$ corresponding to a class $A \in \Cl^{+}$ . The series 
  \begin{equation}
    \label{eq:78}
    \partial f_{Q}^{+}(q) = \log_{p}(u_{A}) - \sum_{n = 1}^{\infty}\qty(
    \sum_{\substack{(Q,\gamma) \in M(n,A)  \\ Q = \left\langle a,b,c \right\rangle\\ 
        (a,p)=1 } }\log_{p}\qty(\frac{-b+n\sqrt D}{2a}) -
    \sum_{\substack{(Q,\gamma) \in M(n,A[\sqrt D])  \\ Q = \left\langle a,b,c \right\rangle\\ 
        (a,p)=1 } }\log_{p}\qty(\frac{-b+n\sqrt D}{2a}))q^{n},
  \end{equation}
  defines an $r$-overconvergent modular form of weight $2$ and tame
  level $1$ for any $r < p/(p+1)$.
\end{cor}
\begin{proof}
Define $\partial f_{Q}^{+}(q) \defeq \frac{2}{h^{+}} \sum_{\psi \text{ odd}} \bar \psi(A) \partial
f_{\psi}^{+}(q)$, which has the effect of replacing $\psi(\mf a)$ in
\cref{eq:25} with $\1^{*}_{A}[\mf a]$. Being a linear combination of
overconvergent modular forms, it is itself overconvergent of same
weight, level and rate of overconvergence.

Using \cref{prop:QF-ideal-bij}, we can rewrite the series in terms of
$M(n,A)$ and $M(n,A[\sqrt D])$, showing that \cref{eq:78} holds for
non-constant terms.

To compute the constant term of $\partial f_{Q}^{+}(q)$, note that formally,
$u_{\psi} = \sum_{A \in \Cl^{+}}\psi(A)\cdot u_{A}$, so
\begin{equation}
  \label{eq:10}
  \frac{2}{h^{+}} \sum_{\psi \text{ odd}} \bar \psi(A) \cdot u_{\psi}
  = \sum_{A \in\Cl^{+}} \frac{2}{h^{+}} \sum_{\psi \text{ odd}} \bar \psi(A)\psi(A)\cdot
  u_{A} = \sum_{A \in \Cl^{+}} \1^{*}_{A} \cdot u_{A} = u_{A}\cdot u_{A[\sqrt D]}^{-1}.
\end{equation}
The condition $\bar u_{A} = 1/u_{A}$ is equivalent to $u_{A[\sqrt
  D]} = u_{A}^{-1}$, so $\frac{2}{h^{+}} \sum_{\psi \text{
    odd}}\frac{1}{2}\log_{p} u_{\psi} = \log_{p}(u_{A})$. 
\end{proof}

This gives a reasonably efficient algorithm for computing
$\log_{p}u_{A}$:

\begin{algorithm}[H]
  \caption{Algorithm for computing $\log_{p}u_{A}$}
  \DontPrintSemicolon
  \KwIn{A real quadratic field $F = \Q(\sqrt D)$, a rational prime $p$
    inert in $F$, a class $A \in \Cl^{+}$ represented by a quadratic
    form $Q_{0}$, and an integer $N$.}
  \KwOut{$\log_{p} u_{A}$ as an element of $F_{p}$ to
    $p$-adic precision $N$.}
$m \gets p\cdot N$\;
Compute $\{M(n,A)\}_{n\le m}$ using \cref{alg:NR-forms}\;
Compute $\{a_{n}(\partial f^{+}_{Q})\}_{n \le m}$ using \cref{eq:78}\;
$B \gets \mathtt{KatzBasis}(M_{2}^{\dag}(\SL_{2}(\Z))) \bmod{p^{N},q^{m}})$\;
$\log_p u_{A} \gets \mathtt{FindConstTerm}(\{a_{n}\}_{n\le m}, B)$\;
\Return{$\log_p u_{A}\bmod p^{N}$}\label{alg:GS-unit}
\end{algorithm}

The step $\mathtt{KatzBasis}$ is described in step 3 of
\cite[Algorithm 1]{lauder2011}. Roughly speaking, a Katz basis form is the
 ratio of a classical modular form of weight $2 + (p-1)i$ and
$E_{p-1}^{i}$. Computing finitely many of these to sufficiently high
finite precision, these span a subspace of $M^{\dag}_{2}(\SL_{2}(\Z))$ in
which we can uniquely detect $\partial f^{+}_{Q}$. Further details and proofs
can be found in \cite[Chap.~2]{katz1973}.

The function $\mathtt{FindConstTerm}$
first solves a linear system obtained by solving for the higher order
coefficients of $\partial f^{+}_{Q}$ in terms of those in $B$, so that the
constant term of $\partial f^{+}_{Q}$ is a linear combination of the constant
terms of the Katz basis forms.  The
number of terms $m$ computed in the $q$-expansion of
$\partial f^{+}_{Q}$ ensures that it can always be found in the Katz basis
from \cite[Algorithm 1]{lauder2011}, although in practice smaller
values of $m$ are often sufficient.

With a little extra work we can compute the spectral expansion of
$e_{\ord}(\partial f^{+}_{Q})$. To compute the ordinary projection, we use a
trick due to Lauder, which does not seem to be recorded in the
literature. The idea is to compute matrix of $U_{p}$ acting on the
Katz basis $B$ from \cref{alg:GS-unit}, computed to precision
$\dim M_{k'}(\SL_{2}(\Z))$ where
$k' \defeq 2+(p-1)\lfloor N(p+1)/p\rfloor$. Since this approximate basis is
finite, the matrix $U_{p}$ has finite rank. Raising the matrix to the
power $2m$ and applying to the vector defining $\partial f^{+}_{\psi}$ then
gives the ordinary projection. We denote this step by
$\mathtt{OrdinaryProjection}$ below:

\begin{algorithm}[H]
  \DontPrintSemicolon
  \caption{Algorithm for spectral expansion of $e_{\ord}(\partial f^{+}_{Q})$}
  \KwIn{A real quadratic field $F = \Q(\sqrt D)$, a rational prime $p$
    inert in $F$, a character $\psi \colon \Cl^{+} \to \C^{\times}$ and a
    positive integer $m$} \KwOut{The coefficients $\lambda_{0}$ and
    $\lambda_f$ of $e_{\ord}(\partial f^{+}_{\psi})$ as elements of
    $F_{p}$, represented with $p$-adic precision $N$.}
  $m \gets \dim M_{2+(p-1)\lfloor N(p+1)/p\rfloor}(\SL_{2}(\Z))$\;
  Compute $B \bmod{(p^{m},q^{N})}$ and
  $\{a_{n}(\partial f^{+}_{\psi})\}_{n =0}^{N}$ as in \cref{alg:GS-unit}\;
  $G \gets \mathtt{OrdinaryProjection}(\{a_{n}(\partial
  f^{+}_{\psi})\}_{n =0}^{N}, B)$\;
  $M \gets M_{2}(\Gamma_{0}(p))\otimes F_{p}$\;
  \Return{$\mathtt{FindInSpace(G,M)}$}
\end{algorithm}
Here $\mathtt{FindInSpace}(G,M)$ solves for $G = e_{\ord}(\partial f^{+}_{Q})$ in
terms of the eigenbasis for $M_{2}(\Gamma_{0}(p))$ and returns the
corresponding coefficients, which are precisely $\lambda_{0}$ and the
$\lambda_{f}$ for eigenforms $f$. The same algorithm works for $e_{\ord}(\partial
f^{+}_{\psi})$. 

\section{From logarithms to invariants}
\label{sec:gs-units}
While the algorithms in the previous section are fairly
straightforward, recovering the $u_{A}$ from $\log_{p} u_{A}$ and
$P_{f,A}$ from $\lambda_{f}$ is quite involved. In this section how to do
so. We start with the simpler case, namely that of Gross--Stark units.

\subsection{Recovering the Gross--Stark unit from its $p$-adic
  logarithm}
The ``virtual units'' $u_{A}$ are difficult to work with because they
are formal powers of units in $H$. Instead, we use the Brumer--Stark
conjecture and look instead for the (conjectural) element
$\epsilon_{A} \in \O_{H}^{\times}[1/p]$ satisfying
$e \cdot u_{A} = \epsilon_{A} \otimes 1$, where $e \defeq \# \mu(H)$. This property
implies that $\log_{p} u_{A} = \frac{1}{e} \log_{p}\epsilon_{A}$. Note that
while $u$ is determined uniquely by \cref{eq:5} because
$\O^{\times}_{H}[1/p]\otimes \Q$ is torsion-free, $\epsilon$ is only unique up to
roots of unity in $H$. This ambiguity is natural for several reasons:
first, the Brumer--Stark units over $\Q$ constructed in
\cite{gross1981} are Gauss sums, which by definition require a choice
of a root of unity to determine the additive character. Second,
$\epsilon$ being defined only up to torsion in
$\O_{H}^{\times}[1/p]$ mirrors the fact that Stark--Heegner points are
defined up to torsion in $E(H)$.

We can find the exact value of $e$ without computing the unit group of
$\O_{H}$ directly by noting that any root of unity in $H$ will lie in
the \emph{genus field} of $F$, the largest subextension
of $H$ which is abelian over $\Q$. This has the following classical
description:

\begin{prop}[{\cite[Prop.~2.19]{lemmermeyer2000}}]
  Let $F = \Q(\sqrt{D})$, and let $D = D_{1}\cdots D_{t}$ be a
  factorisation of $D$ into prime discriminants, meaning $\pm D_{i}$ is
  a prime power with sign chosen so that if $D_{i}$ is odd, then
  $D_{i}\equiv 1 \bmod{4}$. Then the genus field of $F$ equals $\Q(\sqrt{D_{1}},\ldots, \sqrt{D_{t}})$.
\end{prop}
Since the only quadratic extensions with other roots of unity than
$\pm1$ are $\Q(\sqrt{-1})$ and $\Q(\sqrt{-3})$, we obtain the following:
\begin{cor}\label{cor:genus-field}
  We have $\# \mu(H) > 2$ if and only if either of the following holds:
  \begin{enumerate}
  \item $D \equiv 0 \bmod{3}$, in which case $H$ contains a cube root of
    unity.
  \item $D \equiv 0 \bmod{4}$ and $D/4 \equiv 3 \bmod{4}$, in which case
    $H$ contains $\sqrt{-1}$.
  \end{enumerate}
\end{cor}

The kernel of $\log_{p}$ is much larger than that of the archimedean
$\log$, containing powers of $p$ as well as roots of unity. Passing
from $\log_{p}\epsilon_{A}$ to $\epsilon_{A}$ requires knowing both
$\ord_{\mf P}\epsilon_{A}$ and $\epsilon_{A} \bmod{\mf P}$. We can deal with the
latter by looping through all the roots of unity in $H_{\mf P}$, of
which there are $p^{2}-1$, and test the product separately. This,
along with the computation of the Katz basis, are the main bottlenecks
in the algorithm for large values of $p$.  Certain Stark units modulo
$p$ appear in a recent conjecture of Harris--Venkatesh
\cite{harris2019}, and it would be interesting to see if an analogous
conjecture could describe the mod $\mf P$ reduction of
$u_{A}$.

To find the $\mf P$-valuation, we use a classical theorem due to
C.~Meyer which we now describe. Let $A \in \Cl^{+}$ be a narrow ideal
class, and recall that corresponding partial $L$-series is given by 
\begin{equation}
  \label{eq:20}
  L(s,A) \defeq \sum_{\mf a \leq \O_{F}, \, [\mf a] = A} \frac{1}{N(\mf
    a)^{s}}, \qq{} \Re(s) > 1.
\end{equation}
Let $\zeta_{-}(s,A) \defeq \frac{1}{2}(L(s,A)- L(s,A[(\sqrt{D})]))$. This
is non-zero if and only if $F$ has no unit of negative norm, which is a
running assumption.

Let $\epsilon$ denote the fundamental unit of $F$, necessarily satisfying
$N(\epsilon) = 1$, and fix a representative $\mf a \leq \O_{K}$ for
$A$ with $\Z$-basis $1,w$. Then $\epsilon\cdot \mf a = \mf a$, and so we can find
integers $a,b,c$ and $d$ such that
 \[ \epsilon w  = aw + b \qq{and} \epsilon = cw +d.
 \]
 This is done explicitly in \cref{alg:Meyer}. Since the action of
 $\epsilon$ is invertible and preserves the order of the basis, the matrix
 $\gamma_{A} \defeq \mqty(a & b \\ c & d)$ has determinant $1$. Passing to
 the associated quadratic form $Q = Q_{1}x^{2}+ Q_{2}xy + Q_{3}$ using
 \cref{prop:QF-Id-corr} and
 writing $\epsilon = u + t\sqrt{D}$, a straightforward computation shows that
 \begin{equation}
   \label{eq:12}
   \gamma_{A} = \mqty(t+Q_{2}u & 2Q_{3}u \\ -2Q_{1}u & t-Q_{2}u).
 \end{equation}

 Let $\Phi \colon \SL_{2}(\Z) \to \R$ denote the \emph{Dedekind symbol}
 defined by
 \begin{equation}
   \label{eq:7}
   \Phi\mqty(a & b \\ c & d) \defeq
   \begin{cases}
     b/d \qq{} &\text{if } c = 0, \\
     \frac{a+d}{c} - 12\sgn(c) \cdot s(a,c)\qq{}&\text{if } c \neq 0,
   \end{cases}
 \end{equation}
 where $s(a,c)$ is the \emph{Dedekind sum}
 \begin{equation}
   \label{eq:dedekind_sum}
s(a,c) \defeq \sum_{k=1}^{|c|}\qty(\qty(\frac{ak}{c}))
\qty(\qty(\frac{k}{c})) \qq{for} (a,c) =1, c \neq 0,
\end{equation}
with $((x)) = 0$ if $x \in \Z$ and $((x)) = x - \floor{x} - 1/2$
otherwise. 

By adding a correction term to $\Phi$, Rademacher showed that the
eponymous \emph{Rademacher symbol},
\begin{equation}
  \label{eq:520}
  \Psi(\gamma) \defeq \Phi(\gamma) - 3 \sgn(c(a+d)),
\end{equation}
depends only on the conjugacy class of $\gamma$.

\begin{thm}[Meyer]\label{thm:Meyer}
  Fix a class $A \in \Cl^{+}$, and let $\gamma_{A} \in \SL_{2}(\Z)$ be the
  associated matrix. Then
  \begin{equation}
    \label{eq:690}
    \zeta_{-}(0,A) = \frac{1}{12} \Psi(\gamma_{A}).
  \end{equation}
\end{thm}
A proof can be found in \cite[\S 2.45]{siegel1961}.

\begin{cor}\label{cor:GS-pval}

  Let $u_{A}$ be a Gross--Stark unit attached to a narrow ideal class
  $A$. Then 
  \begin{equation}
    \label{eq:750}
    \ord_{\mf P} u_{A} = -\frac{1}{12} \Psi(\gamma_{A}).
  \end{equation}
  Similarly, for the associated Brumer--Stark unit $\epsilon_{A}$,
  \begin{equation}
    \label{eq:11}
    \ord_{\mf P} \epsilon_{A} = -\frac{e}{12} \Psi(\gamma_{A}),
  \end{equation}
  where $e = \#\mu(H)$. 
\end{cor}

\begin{proof}
  By \cref{eq:5},
  \begin{subequations}
    \begin{align}
      \ord_{\mf P} u_{A}
      &= \frac{1}{2}(\ord_{\mf P} u_{A}-\ord_{\mf P} u_{A[\sqrt{D}]})\\
      &= -\frac{1}{2}(L(0,A) - L(0,A[\sqrt D]) \\
      &= - \zeta_{-}(0,A) = - \frac{1}{12}\Psi(\gamma_{A}). 
    \end{align}
  \end{subequations}
  The second claim follows immediately from the identity $e\cdot u_{A} = \epsilon_{A}\otimes 1$.
\end{proof}

We then have the following algorithm for computing $\ord_{\mf P}
\epsilon_{A}$:

\begin{algorithm}[H]
\DontPrintSemicolon
  
\KwInput{
  An indefinite binary quadratic form $Q(x,y) = Q_{1}x^{2}+Q_{2}xy+Q_{3}y^{2}$
    of square-free discriminant $D$, representing a narrow ideal class $A$  of $F = \Q(\sqrt{D})$.
}\;
\KwOutput{$\ord_{\mf P} \epsilon_{A}$}\;

$t,u \gets$ {\tt PellSolution}$(D)$ \tcp*{Solve Pell's equation in
  $\Q(\sqrt D)$ to find fundamental unit $\epsilon = u+t\sqrt{D}$.}\;
$\gamma_{A} \defeq \mqty(a & b \\ c & d) \gets \mqty(t+Q_{2}u & 2Q_{3}u \\ -2Q_{1}u & t-Q_{2}u)$ \;
\If{$c = 0$}
{$\Phi \gets b/d $}
\Else
{$\Phi \gets \frac{a+d}{c} - 12\sgn(c)\cdot\mathtt{DedekindSum}(a,c)$}

$\Psi \gets \Phi - 3 \sgn(c(a+d))$\;
\Return{ $-e\cdot \Psi/12$}
\caption{Compute $\ord_{\mf P} \epsilon_{A}$ using Meyer's formula\label{alg:Meyer}}  
\end{algorithm}

The fundamental solution of Pell's equation grows very quickly as $D$
gets large, so computing Dedekind sums by evaluating
\cref{eq:dedekind_sum} directly can be very slow for large values of
$D$. Instead we use a formula from \cite[Ex.~3.10]{apostol1990}: By
replacing $c$ by $-c$ and $a$ by $a \bmod c$, we can assume that
$0 < a < c$.  Let $r_{0} \defeq c$, $r_{1} \defeq a$ and define
$r_{j}$ recursively to be the remainders in the Euclidean algorithm
applied to $a$ and $c$, satisfying
$r_{j+1} \equiv r_{j-1} \bmod r_{j}$ and
$1 = r_{n+1} < \ldots r_{j+1} < r_{j} \ldots < r_{0}$ for all
$1 \le j \le n-1$. Then
\begin{equation}
  \label{eq:61}
  s(a,c) = \frac{1}{12}
  \sum_{j=1}^{n+1}\qty(\frac{r^{2}_{j}+r_{j-1}^{2}+1}{r_{j}r_{j-1}}) -
  \frac{(-1)^{n}+1}{8}.
\end{equation}
This is very efficient in practice.

We make the convention of calling the minimal polynomial of $\epsilon$
the irreducible polynomial $P$ satisfying $P(\epsilon) = 0$ of minimal degree with
coefficients in $\O_{F}$ not all divisible by the same prime, such
that the leading term is a positive power of $p$.

\begin{lemma}\label{lemma:GS-props}
  Let $\epsilon$ be a Brumer--Stark unit in $\O_{H}[1/p]^{\times}$, and let $P(T) = \sum_{i=0}^{d}a_{i}T^{i} = a_{d}\prod_{\sigma \in G}(T-\sigma(\epsilon))$ be its
  minimal polynomial. Then
\begin{enumerate}
\item $\epsilon$ is a primitive element of $H$ over $F$, $H = F(\epsilon)$. 
\item $P$ is of degree $h^{+}$, and after possibly twisting $\epsilon$ by a root
  of unity in $H$, has rational coefficients.
\item $P$ is reciprocal, $a_{i} = a_{d-i}$ for all $0 \le i \le d$.
\end{enumerate}
\end{lemma}

\begin{proof}
  (i) We follow the strategy of \cite[Théorème 2.3]{roblot1997}.
  Suppose $\sigma(\epsilon) = \epsilon$ for some
  $\sigma \in G$. For any character $\chi \colon G \to \C^{\times}$, let
  $ L_{S}(s,\chi)$ denote the $L$-function of $\chi$ with the Euler factor
  at $\mf P$ removed. Since $\sigma_{\mf P}=1$,
  $\chi(\sigma_{\mf P}) = 1$, and so we have $L_{S}(0,\chi) = 0$.  A consequence of the
  Brumer--Stark conjecture, see for example
  \cite[Prop.~(5.5) and Conj.~(4.2)]{tate1982}, is that $\epsilon$ satisfies
\begin{equation}
    \label{eq:stark-conj}
    L_{S}'(0,\chi) = -\frac{1}{e} \sum_{\sigma' \in G}\chi(\sigma')\log |\sigma'(\epsilon)|_{\mf P}
  \end{equation}
  for all $\chi$. It follows that
  \begin{subequations}
    \begin{align}
       L'_{S}(0,\chi) &= - \frac{1}{e} \sum_{\sigma' \in G}\chi(\sigma')\log |\sigma'(\epsilon)|_{\mf P}
      \\
              &= - \frac{1}{e} \sum_{\sigma' \in G}\chi(\sigma')\log |\sigma'\sigma(\epsilon)|_{\mf P}\\
              &= - \frac{\bar \chi(\sigma)}{e} \sum_{\sigma'' \in G}\chi(\sigma'')\log
                |\sigma''(\epsilon)|_{\mf P}\\
              & = \bar \chi(\sigma) L_{S}'(0,\chi).
    \end{align}
  \end{subequations}
  If $\chi$ is odd, then $L'_{S}(0,\chi) \ne 0$ by
  \cite[Eq.~3.1]{gross1981}, so
  $\sigma \in \bigcap_{\chi \text{ odd}}\ker \chi$. Fix an odd character
  $\psi$, and note that there is a bijection between even characters
  $\chi$ and the set of characters $\psi\cdot \psi'$ where $\psi'$ runs over all odd
  characters. Now
  \begin{equation}
    \label{eq:14}
    \sum_{\chi \in \hat G} \chi(\sigma) = \sum_{\chi \text{ odd}}\chi(\sigma) + \sum_{\chi \text{
        even}}\chi(\sigma) = (1 + \psi(\sigma))\sum_{\psi' \text{ odd}}\psi'(\sigma) = 2\#\{ \chi \text{
      odd}\} = h^{+},
  \end{equation}
  and so $\sigma = 1$.
  
  (ii) The degree of $P$ is $h^{+}$ since $\epsilon$ is primitive. Let
  $\sigma_{\mf P}\in \Gal(H/\Q)$ be the Frobenius at $\mf P$. If we lift any
  $\sigma \in G$ to $\Gal(H/\Q)$, then
  $\sigma_{\mf P}\sigma = \sigma^{-1}\sigma_{\mf P}$. It follows that
  $\ord_{\mf P}\sigma_{\mf P}(\epsilon) = \frac{1}{e} L(0,A_{\sigma^{-1}})$. Now note
  that $L(0,A_{\sigma}) = L(0,A_{\sigma^{-1}})$ for all $\sigma$; this can be seen by
  using \cref{cor:GS-pval} and noting that on quadratic forms,
  $A \mapsto A^{-1}$ is given by the map $Q(x,y) = Q(-x,y)$ on associated
  quadratic forms. It follows that
  $\sigma_{\mf P}(\epsilon) = \zeta \epsilon$ for some
  $\zeta \in \mu(H)$.

  Now $\sigma_{\mf P}(\epsilon)$ has the right valuations to be a Brumer--Stark
  unit, and it suffices to show that $\sigma_{\mf P}(\epsilon)^{1/e}$ generates an
  abelian extension of $F$. By \cite[Prop.~IV.1.2, $(d)\Rightarrow(a)$]{tate1984}, it suffices to show that there exists a system of
  elements $\{\alpha_{i}\}$ corresponding to generators
  $\{\sigma_{i}\}$ of $G$ such that
  $\sigma_{\mf P}(\epsilon)^{\sigma_{i}-n_{i}} = \alpha_{i}^{e}$, where
  $n_{i} \defeq \chi_{\mathrm{cyc}}(\sigma_{i})$. For simplicity, we use
  Tate's convention of writing the left Galois action as a
  superscript. Pick a system $\{\tilde \alpha_{i}\}_{i}$ for $\epsilon$, so that 
  \begin{equation}
    \label{eq:13}
    ( \sigma_{\mf P}(\epsilon))^{\sigma_{i}-n_{i}} = \zeta^{\sigma_{i}-n_{i}} \epsilon^{\sigma_{i} - n_{i}}
    = \epsilon^{\sigma_{i}-n_{i}} = \tilde \alpha_{i}^{e}.
  \end{equation}
  Now as in \cite[\S IV.3.7]{tate1984}, we can choose a root of unity
  such that $\epsilon$ has the same class as $\sigma_{\mf P}(\epsilon)$ modulo $e$-th
  powers, and this implies that $\zeta = 1$.
  
  (iii) $P$ being reciprocal is equivalent to
  $P(T) = T^{d}P(1/T)$, which is true if for any non-zero
  root $v$ of $P$, $1/v$ is also a root of $P$. But if $\tau$ denotes complex
  conjugation in $G$, then by \cref{eq:5}, $\tau(\sigma(\epsilon)) =
  1/\sigma(\epsilon)$.
\end{proof}

We can use the knowledge of the $\mf P$-valuations of all the
conjugates of $\epsilon$ to get bounds on the coefficients of $P$ using
the Newton polygon.

\begin{lemma}\label{lemma:GS-polygon}
   Let $v_{0}\ldots,v_{d/2-1}$ be the $\mf P$-valuations of the conjugates of
  $\epsilon$ which are positive, ordered so that
  $v_{0} \ge v_{1} \ge \ldots \ge v_{d/2-1}\ge 0$, and $v_{d/2} =0$. Then for any
  $i = 0,\ldots,d/2 $ we have
  $\ord_{p}(a_{i}) \ge \sum_{j = 0}^{d/2-i} v_{d/2-j}$. In particular,
  $\ord_{p}(a_{d}) = \ord_{p}(a_{0}) = \sum_{j=1}^{d/2}v_{j}$.
\end{lemma}  
\begin{proof}
  By \cref{lemma:GS-props} (iii), the Newton polygon of $P$ is
  symmetric around the vertical line $x = d/2$, and its slopes are
  precisely equal to the $p$-valuations of the roots of $P$, the
  conjugates of $u$. Since $P$ is normalised, we know that
  $\ord_{p}a_{d/2} = 0$, so the Newton polygon of $P$ intersects the
  $x$-axis in the point $(0,d/2)$. To estimate the remaining
  coefficients, note that the Newton polygon of $P$ will always lie in
  the convex hull of the polygon determined as follows: the boundary
  is symmetric around the line $x = d/2$, and is determined by the
  points $(i,\sum_{j=0}^{d/2-i}v_{j})$ for $0 \le i\le d/2$. Since the
  $y$-coordinate of a point determining the Newton polygon of $P$ is
  the $\mf P$-valuation of the corresponding coefficient, this gives the
  required inequality.
\end{proof}

  \begin{figure}[h!]
    \centering
\begin{tikzpicture}
  \draw[->] (-1,0) -- (7,0) node[right] {$x$};
  \draw[->] (0,-1) -- (0,6) node[above] {$y$};
  \draw[-,dotted] (3,-1) -- (3,5);
  
  \filldraw[black] (0,5) circle (2pt) node[anchor= north east] {(0,5)}
  -- (1,2);
  \filldraw[black] (1,2) circle (2pt) node[anchor= east] {(1,2)} -- (2,1);
  \filldraw[black] (2,1) circle (2pt) node[anchor= east] {(2,1)} -- (3,0);
  \filldraw[black] (3,0) circle (2pt) node[anchor= north west] {(3,0)} -- (4,1);
  \filldraw[black] (4,1) circle (2pt) node[anchor= west] {(4,1)} --
  (5,2);
  \filldraw[black] (5,2) circle (2pt) node[anchor= west] {(5,2)} --
  (6,5);
  \filldraw[black] (6,5) circle (2pt) node[anchor= north west] {(6,5)};
\end{tikzpicture}
\caption{The largest possible Newton polygon determined by the
  $\mf P$-valuations of the conjugates of a Gross--Stark unit
  over $\Q(\sqrt{321})$, where the vector of valuations is given by $(-3,-1,-1,1,1,3)$.}
\end{figure}
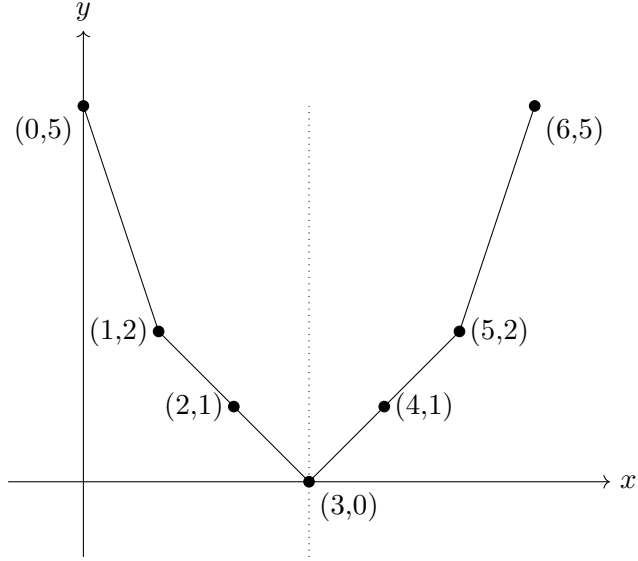

Let $\alpha = (\alpha_{1},\alpha_{2}) \in \Z/p^{m} \times \Z/p^{m}$ be an approximation of $\exp_{p}(\log_{p} \epsilon_{A} ) $, where for a
fixed generator $s$ of $\Q_{p^{2}}$ over $\Q_{p}$ we define the
natural map
\begin{equation}
  \label{eq:49}
  \Z_{p^{2}} = \Z_{p}[s] \to \Z/p^{m} \times \Z/p^{m} \qq{by} a+bs \mapsto (a \bmod{p^{m}}
  ,b \bmod{p^{m}}). 
\end{equation}
To find the minimal polynomial $P$ of $\alpha$, we apply the
\texttt{LLL} algorithm to look for linear integral relations between
powers of $\alpha$. This a common application of lattice reduction
algorithms, and a more detailed exposition can be found in \cite[\S
2.7.2]{cohen1993}. Roughly speaking, the \texttt{LLL} algorithm takes
as input a basis $b_{1},\ldots, b_{d}$ for a Euclidean lattice
$\Lambda \subset \R^{n}$, and returns a ``better'' basis
$b_{1}^{*},\ldots, b_{d}^{*}$ for $\Lambda$, in the sense that
$b_{1}^{*}$ has relatively small norm and that the vectors are
approximately orthogonal. Let $v_{0},\ldots, v_{d/2-1}$ be the
$\mf P$-valuations of the conjugates of $\epsilon_{A}$ ordered as in
\cref{lemma:GS-polygon}, computed using \cref{alg:Meyer}. We want to
find a short vector in the lattice spanned by the rows of the
following $(d/2+3) \times (d/2+3)$-matrix:

\begin{equation}
  \label{eq:42}
	 \begin{pmatrix} 1 & 0 &  \ldots & 0 & p^{v_{0}}(1+\alpha^{d})_{1} & p^{v_{0}}(1+\alpha^{d})_{2} \\
	     0& 1 & \ldots & 0 & p^{v_{1}}(\alpha^1 + \alpha^{d-1})_1&
                            p^{v_{1}}(\alpha^1+\alpha^{d-1})_2 \\
	     0& 0 & \ldots & 0 & p^{v_{2}}(\alpha^2 + \alpha^{d-2})_1 & p^{v_{2}}
                            (\alpha^2+\alpha^{d-2})_2 \\
\vdots& \vdots & \ddots& \vdots & \vdots & \vdots 	    \\
	    0 & 0 &  \ldots & 1 & (\alpha^{d/2})_1 &     (\alpha^{d/2})_2 \\
	    0 & 0 & \ldots & 0 & p^{m}  &              0 \\
	    0 & 0 & \ldots   & 0 &        0&          p^{m} \end{pmatrix}
\end{equation}

A vector
\begin{align}
  \label{eq:51}
  w = &\Big(n_{0},\ldots, n_{d/2}, n_{d/2}\alpha^{d/2}_{1}  + \sum_{i=0}^{d/2-1}p^{v_{i}}n_{i}(\alpha^{i}+\alpha^{d-i} +
p^{m})_{1},\\
      &n_{d/2}\alpha^{d/2}_{2}+\sum_{i=0}^{d/2-1}p^{v_{i}}n_{i}(\alpha^{i}+\alpha^{d-i}+
     p^{m})_{2}\Big), \nonumber
\end{align}
in the lattice is small only if
$n_{d/2}\alpha^{d/2}+\sum_{i=0}^{d/2-1}p^{v_{i}}n_{i}(\alpha^{i}+\alpha^{d-i}) \equiv 0
\bmod{p^{m}}$. Then the polynomial
$\sum_{i=0}^{d/2}p^{v_{i}}n_{i}x^{i} +
\sum_{i=d/2+1}^{d}p^{v_{d/2-i}}n_{d-i}x^{i}$ is a good candidate for the
minimal polynomial of $P$ over $\Q$. This suggests the following
algorithm:

\begin{algorithm}[H]
\DontPrintSemicolon
\KwInput{
  \begin{itemize}[itemsep=-4pt]
  \item $\alpha \in \Q_{p^{2}}$ an approximation to
    $\exp_{p}(\log_{p}\epsilon_{A})$,
  \item $v_{0},\ldots, v_{d/2-1}$ as in \cref{lemma:GS-polygon}.
\end{itemize}}
  \KwOutput{The minimal polynomial $P \in \Z[x]$ of $\epsilon_{A}$.}\;
  $\zeta \gets $ primitive $(p^{2}-1)$-st root of unity in $\Q_{p^{2}}$ \;
  \For{$k=0$ to $p^{2}-1$}    
  {
    $\alpha' \gets \zeta^{k}\alpha$ \;
    $M \gets$ matrix described in \cref{eq:42} with $\alpha'$ in place of $\alpha$ \;
    $v = (n_{i}) \gets $ first vector returned by $\mathtt{LLL}(M)$ \;
    $P \gets \sum_{i=0}^{d/2}n_{i}x^{i} + \sum_{i=d/2+1}^{d}n_{d-i}x^{i}$ \;
    \If{$n_{0} = p^{r}$ for some $r \in \N$}
    {    \If(\tcp*[f]{Described below}){$\mathtt{IsGSUnitCharPoly}(P)$}
    {\Return $P$.}}
}
 \Return 0 \;
\caption{Find the minimal polynomial of $\epsilon_{A}$ from $p$-adic
  approximation of $\log_{p} \epsilon_{A}$\label{alg:rec-algdep}}
\end{algorithm}
In practice, it is convenient to pick $A \in \Cl^{+}$ so that
$\ord_{\mf P} \epsilon_{A}$ is as close to $0$ as possible. 

The function \texttt{IsGSUnitCharPoly} performs a series of test in
order, and returns \texttt{False} if any test fails:
\begin{enumerate}
\item if $P$ is irreducible over $F$, hence generates an extension of
  $F$ of degree $h^{+}$,
\item if the absolute discriminant of $H' \defeq F[x]/(P(x))$ is a power of $D$,
  which is equivalent to $H'/F$ being unramified at all finite places,
\item if $H'/F$ is abelian. \\
  At this point we know that $H' \cong H$,
  but to ensure that $P$ is the minimal polynomial of a Gross--Stark
  unit and not just any generator of $F$, we perform a further test:
\item test if the extension generated by $P(x^{e})$ is a central extension. 
\end{enumerate}
If all of these tests are passed, then it is quite likely, although
not absolutely certain, that the polynomial $P$ has a Brumer--Stark
unit as a root. To be absolutely certain, one should test if
$P(x^{e})$ generates an abelian extension of $F$, and check
numerically that the roots satisfy \cref{eq:stark-conj}, using for
example Dokchitser's $L$-functions calculator. However, this is
computationally demanding when both $e$ and $[H:F]$ are
large.
\begin{remark}
  The requirement that the extension should be central was part of
  Stark's original conjecture, see \cite[Conj.~1]{stark1980}, and on
  \cite[p.~40]{popescu2011} Stark notes that this was sufficient
  for the factorisation of regulators which motivated it. The
  condition that the extension should be abelian likely arose from
  Tate's work leading to the formulation of the Brumer--Stark
  conjecture, and is now known to be true. It would be interesting to
  know whether ``central implies abelian'' in this situation, that is:
  if $\alpha$ is a $p$-unit which generates $H$ with $\mf P^{\sigma}$-valuations
  specified by \cref{eq:5} and $\sqrt[e]{\alpha}$ generates a central
  extension of $F$, is the extension actually abelian?
\end{remark}

To describe the test in (iv), it is convenient to introduce some
notation: Let $K \defeq H(\sqrt[e]{\epsilon_{A}})$ and
$G_{e} \defeq \Gal(K/H)$. By Kummer theory, $G_{e} \cong \Z/e\Z$. In this case
$\Gamma \defeq \Gal(K/F)$ is a group extension of $G_{e}$ and $G$,
\begin{equation}
  \label{eq:66}
  1 \to G_{e} \to \Gamma \to G \to 1.
\end{equation}
The following lemma gives a simple criterion for deciding whether
$\Gamma$ is a central extension, that is, if $G_{e}$ lies in the centre of
$\Gamma$, without computing $\Gamma$ directly:
\begin{lemma}\label{lemma:central-ext}
  Let $F$ be a number field, $H/F$ a Galois extension containing all $e$-th
  roots of unity, and $\alpha \in H^{\times}$. Define $\chi_{\cyc}\colon G \defeq
  \Gal(H/F)\to (\Z/e\Z)^{\times}$ by $\zeta^{\chi_{\cyc}(\sigma)} = \sigma(\zeta)$ for any $\zeta \in \mu_{e}(H)$. Then
  $K \defeq H(\sqrt[e]{\alpha})/F$ is a central extension if and only if for all $\sigma
  \in G$ there exists some $\beta \in H^{\times}$ such that
  $\sigma(\alpha) = \alpha^{\chi_{\cyc}(\sigma)}\beta^{e}$. 
\end{lemma}
\begin{proof}
  There is a natural action of $G$ on $G_{e} \defeq \Gal(K/H)$ by
  conjugation, $\sigma \cdot g \defeq \sigma g\sigma^{-1}$, which is well-defined
  precisely because $G_{e}$ is abelian. The extension $K/F$ is central
  if and only if the action is trivial. Let $\Delta$ be a set of
  representatives of $H^{\times}/(H^{\times})^{e}$, and note that this admits a
  natural action of $G$. The Kummer pairing (\cite[\S I.6]{gras2003})
  gives a $G$-equivariant isomorphism
  $G_{e} \cong \Hom(\Delta, \mu_{e}(K))$. The action of $G_{e}$ on the right-hand
  side is given by
  $(\sigma\cdot \phi)(\alpha) = \phi(\sigma^{-1}(\alpha))^{\chi_{\cyc}(\sigma)}$ where
  $\chi_{\cyc}(\sigma)$ is defined by
  $\sigma \cdot \zeta_{e} = \zeta_{e}^{\chi_{\cyc}(\sigma)}$. The action of
  $G$ on $G_{e}$ is trivial if and only if the action on
  $\Hom(\Delta, \mu_{e})$ is. Each element of this group is given by
  $\psi_{g} \colon \delta \mapsto \langle \delta,g \rangle \defeq \frac{g\sqrt[e]{\delta}}{\sqrt[e]{\delta}}$
  for some $g \in G_{e}$, and so $\Gamma$ is central if and only if
  $(\sigma \cdot \psi_{g})(\delta) = \psi_{g}(\delta)$ for all
  $\delta \in \Delta$, $g \in G_{e}$ and $\sigma \in G$. Equivalently,
  \begin{equation}
    \label{eq:15}
    \qty(\frac{g\sqrt[e]{\sigma^{-1}(\delta)}}{\sqrt[e]{\sigma^{-1}(\delta)}})^{\chi_{\cyc}(\sigma)}
    =\frac{g\sqrt[e]{\delta}}{\sqrt[e]{\delta}} \qq{hence}
    g\qty(\sqrt[e]{\frac{\alpha^{\chi_{\cyc}(\sigma)}}{\sigma(\alpha)}}) = \sqrt[e]{\frac{\alpha^{\chi_{\cyc}(\sigma)}}{\sigma(\alpha)}},
  \end{equation}
  where $\alpha \defeq \sigma^{-1}(\delta)$. This being true for all $g$ is
  equivalent to $\frac{\alpha^{\chi_{\cyc}(\sigma)}}{\sigma(\alpha)}$ being an $e$-th power for all
  $\sigma$. Finally, note that $G$ acts transitively on $\Delta$, so it
  suffices to check the criterion for a single $\alpha$. 
\end{proof}

This test can be implemented quite easily, and is mainly bottlenecked
by the computation of $\Gal(H/F)$, at least when $[H:F]$ is reasonably large.


\begin{remark}
  A test for whether an extension is abelian is found in
  \cite[Algorithm 4.4.6]{cohen2012}. In short, the Takagi existence
  theorem gives a bijection between abelian extensions $K/F$ and
  certain \emph{Takagi subgroups} of a ray class group
  $\Cl_{\mf m} F$, where $\mf m$ is a sufficiently large
  modulus. However, this is very slow when $e$ and $h$ are large,
  because it requires computing the ray class group of $F$ of modulus
  equal to the relative discriminant of $H(\sqrt[e]{\alpha})/F$, which is
  relatively large. 
\end{remark}

\subsection{Detecting Stark--Heegner points}
\label{sec:SH}
Our method of finding Stark--Heegner points is much more
primitive, because we don't have an equivalent of the Brumer--Stark
conjecture. 

Let $E/\Q$ be an elliptic curve with split multiplicative reduction at
$p$. Recall from \cref{thm:eigenspansion} that if $E$ has associated
eigenform $f \in M_{2}(\Gamma_{0}(p))$, then the corresponding spectral
coefficient $\lambda_{f} = -L_{\mathrm{alg}}(1,f)\log_{E}(P_{\psi,f})$ involves
a point on $P_{\psi,f}$ conjecturally defined over $H$. To find this, we
make use of the Tate curve $E_{q}$ isomorphic to $E$, which is
described explicitly with the formulae in \cite[\S
C.14]{silverman2009}. From this we can find an explicit isomorphism
$F_{p}^{\times}/q^{\Z} \xrightarrow{\phi} E_{q}(F_{p})$, where $q$ is an
element satisfying $|q| < 1$ generating a discrete subgroup.  An
approximation to
$\alpha \defeq \exp_{p}(-\lambda_{f}/L_{\mathrm{alg}}(1,f))$ can then be mapped
to a point on the Tate curve $E_{q}(F_{p})$. Mapping further into
$E(F_{p})$, we a compute using descent a generating set $\{g\}$ for
$E(H)$ and attempt to write the image of $\alpha$ as an integral
combination of them. Since $P_{\psi,f}$ is only defined up to torsion, it
is reasonable to look for a dependence between the formal logarithms
of $\alpha$ and the generators $\{g\}$. To ensure convergence of the
corresponding power series, we replace $\alpha$ by $\alpha^{p-1}$ and each $g$ by
$(p-1)g$. Then we look for an integer relation by applying the
\texttt{LLL}-algorithm to a suitable lattice as in the previous section.
Following the convention in \texttt{pari/gp}, we call this step \texttt{lindep}.

In summary, we have the following algorithm:

\begin{algorithm}[H]
\DontPrintSemicolon
\KwInput{
  \begin{itemize}[itemsep=-4pt]
  \item A normalised eigenform $f$ in $M_{2}(\Gamma_{0}(p))$ with Hecke
    field $\Q$,
  \item an elliptic curve $E$ with associated eigenform $f$,
  \item $\lambda_{f} \in (\Z/p^{m}\Z)^{2}$ an approximation to
    $- L_{\mathrm{alg}}(1,f) \log_{E_{f}}(P_{\psi,f}) \in F_{p}$,
\end{itemize}}
  \KwOutput{The point $P_{\psi,f}$ on the elliptic curve $E$}\;
  $E_{q} \gets \mathtt{TateCurve}(E)$ \tcp*{Using formulae in \cite[\S C.14]{silverman2009}}
  $\phi \gets \mathtt{Isomorphism}(F_{p}^{\times}/q^{\Z}, E_{q})$ \tcp*{As in \cite[Thm.~14.1]{silverman2009}}
  $\beta \gets \phi(-\lambda_{f}/L_{\mathrm{alg}}(1,f))$\;
  $H \gets \mathtt{NarrowHilbertClassField}(F)$\;
  $E(H) \gets \mathtt{MordellWeilGroup}(E/H)$\;
  
  $L \gets [\log_{E_{q}}((p-1)\beta)]$\;
  \tcp{Compute formal logarithms of
        non-torsion generators of $E(H)$:}
  \For{$g \in \mathtt{Generators}(E(H))$}{
    \If{$\mathtt{Order(g) == 0}$}{
      $L \gets L \cup \{ \log_{E} ((p-1)g)\}$ \; 
    }
  }
  $(n_{1},(n_{g})) \gets \mathtt{lindep}(L)$ \tcp*{Find integer relation between formal
    logarithms using \texttt{LLL}.}
 \Return $\sum_{g} n_{g}\cdot g /n_{1} \in E(H)$ \;
\caption{Find Stark--Heegner point $P_{\psi,f}$ from $\lambda_{f}$ \label{alg:stark-heegner}}
\end{algorithm}
By linearity, the algorithm works equally well when $\lambda_{f}$ comes from
$\partial f^{+}_{Q}$, in which case the corresponding Stark--Heegner point is
a weighted sum of points $P_{\psi,f}$. The rational part of the
$L$-value can be computed either directly in \texttt{magma} using the
intrinsic \texttt{LRatio}, or by using the BSD formula and the
invariants of $E$ since $L(s,f) = L(s,E)$, or even analytically by
approximating $L(1,E)$ and computing the period integrals of $E$.

One limitation of algorithm \cref{alg:stark-heegner} is that computing
$E(H)$ is very slow when $[H:\Q] > $. We hope to resolve this this in
the future by improving the algorithms for detecting polynomials from
$p$-adic approximations to the their roots. 

In the table below we have computed the minimal polynomials of the $X$
and $Y$ coordinates of the Stark--Heegner points coming from
$\partial f^{+}_{\psi}$ on the curve
$E: y^2 + xy + y = x^3 - x^2 - x - 14$. This is a model for
$X_{0}(17)$, for which we have $L_{\mathrm{rat}}(1,f) = 1/4$, so
$\lambda_{f} = -\frac{1}{4}\log_{E}P_{\psi,f}$.

Here $\psi$ runs over each genus character associated with $D$. Since all
the fields $\Q(\sqrt D)$ for $D < 100$ with no fundamental unit of
negative norm such that $\qty(\frac{D}{17}) = -1$ have narrow class
number $2$, there is a unique nontrivial character. This satisfies
$\partial f^{+}_{\psi} = -\partial f^{+}_{Q}$ where $Q$ is a quadratic form with class
corresponding to the inverse different in $\Cl^{+}$.  Note that this
matches the table on p.~545 of \cite{darmon2021}.

\begin{table}[H]
  \centering
  \def\arraystretch{1.2}
\begin{tabular}{c|c|c}
$D$ & $X$ & $Y$ \\ \hline
$12$ & $x^{2} - 6 x + 10$ & $x^{2} - 2 x + 10$ \\
$24$ & $x^{2} + \frac{2}{9} x + \frac{89}{9}$ & $x^{2} + \frac{230}{27} x + 25$ \\
$28$ & $x^{2} - 6 x + 10$ & $x^{2} + 10 x + 41$ \\
$44$ & $x^{2} - 14 x + 338$ & $x^{2} - 26 x + 7394$ \\
$56$ & $x^{2} + \frac{2}{9} x + \frac{89}{9}$ & $x^{2} + \frac{230}{27} x + 25$ \\
$57$ & $x^{2} + \frac{2306}{1225} x + \frac{6521}{1225}$ & $x^{2} + \frac{111042}{42875} x + \frac{15319}{8575}$ \\
$88$ & $x^{2} + \frac{2}{9} x + \frac{89}{9}$ & $x^{2} - \frac{182}{27} x + \frac{401}{9}$ \\
$92$ & $x^{2} - 6 x + 10$ & $x^{2} - 2 x + 10$ \\
\end{tabular}
\caption{Table of Stark--Heegner points on
  $E: y^2 + xy + y = x^3 - x^2 - x - 14$, for $D<100$.}
  \label{tab:SH-17}
\end{table}

\subsection{Tables of Brumer--Stark units}
Below we show some tables of minimal polynomials of Brumer--Stark units
in different ranges. Full tables are in the author's \texttt{github}
repository, \url{https://github.com/havarddj/drd}.

\begin{table}[H]
  \centering
\begin{tabular}{c|c||c|c||c|c}
$D$ & $P_{D}$ & $D$ & $P_{D}$ & $D$ & $P_{D}$ \\ \hline
$44$ & $3x^{2} + 5x + 3$ & $152$ & $3x^{2} + 2x + 3$ & $236$ & $27x^{2} + 5x + 27$ \\
$56$ & $3x^{2} + 2x + 3$ & $161$ & $27x^{2} + 38x + 27$ & $248$ & $27x^{2} - 46x + 27$ \\
$77$ & $3x^{2} + 5x + 3$ & $188$ & $243x^{2} - 298x + 243$ & $284$ & $2187x^{2} - 4090x + 2187$ \\
$92$ & $27x^{2} + 38x + 27$ & $209$ & $3x^{2} + 5x + 3$ & $305$ & $9x^{4} + 5x^{3} + 17x^{2} + 5x + 9$ \\
$140$ & $81x^{4} + 6x^{3} - 149x^{2} + 6x + 81$ & $221$ & $9x^{4} - 2x^{3} - 5x^{2} - 2x + 9$ & $329$ & $243x^{2} - 298x + 243$ \\
\end{tabular}
  \caption{Minimal polynomials of Brumer--Stark units for $p=3$, $D < 330$.}
  \label{tab:BS-units-3}
\end{table}

\begin{table}[H]
  \centering
  \def\arraystretch{1.2}
\begin{tabular}{c|c}
$D$ & $P_{D}$ \\ \hline
$2005$ & $2^{12}x^{8} + 2^{4} \cdot 1055x^{7} + 2^{2} \cdot 9419x^{6} + 57995x^{5} + 66831x^{4} + 57995x^{3} + 2^{2} \cdot 9419x^{2} + 2^{4} \cdot 1055x + 2^{12}$ \\
$2013$ & $2^{30}x^{4} - 2^{3} \cdot 57677665x^{3} - 1118365527x^{2} - 2^{3} \cdot 57677665x + 2^{30}$ \\
$2021$ & $2^{9}x^{6} + 2^{2} \cdot 111x^{5} + 2^{1} \cdot 123x^{4} - 101x^{3} + 2^{1} \cdot 123x^{2} + 2^{2} \cdot 111x + 2^{9}$ \\
$2037$ & $2^{18}x^{4} + 2^{3} \cdot 16215x^{3} - 263887x^{2} + 2^{3} \cdot 16215x + 2^{18}$ \\
$2045$ & $2^{6}x^{4} - 9x^{3} - 65x^{2} - 9x + 2^{6}$ \\
$2077$ & $2^{3}x^{2} + 15x + 2^{3}$ \\
$2085$ & $2^{24}x^{4} - 2^{3} \cdot 6289393x^{3} + 70333881x^{2} - 2^{3} \cdot 6289393x + 2^{24}$ \\
$2093$ & $2^{8}x^{4} - 2^{1} \cdot 217x^{3} + 645x^{2} - 2^{1} \cdot 217x + 2^{8}$ \\
$2101$ & $2^{13}x^{6} + 2^{6} \cdot 79x^{5} - 2^{3} \cdot 1009x^{4} - 10161x^{3} - 2^{3} \cdot 1009x^{2} + 2^{6} \cdot 79x + 2^{13}$ \\
\end{tabular}
  \caption{Minimal polynomials of Brumer--Stark units for $p=2$,
    $2000\le D \le 2101$.}
  \label{tab:BS-units-2}
\end{table}

The coefficients of the polynomials are all of roughly the same
magnitude, despite the strong conditions on the $p$-valuation of the
constant terms. In particular, the logarithmic height of the middle
coefficient is roughly $\ord_{p}(a_{0})$, which is easily computed in
terms of $L$-values using \cref{eq:5}. A classical result of Schur
says that the coefficients of cyclotomic polynomials can be
arbitrarily large. It would be interesting to know whether the
same holds for our polynomials, normalised to be monic. The largest
value we find is $822.637$, across the tables for
$p \in \{2,3,5,7,11\}$. The following plot shows the absolute value of the middle
coefficient of the normalised polynomials against the discriminant for
different $p$:

\begin{figure}[H]
  \centering
  \includegraphics[width=.7\linewidth]{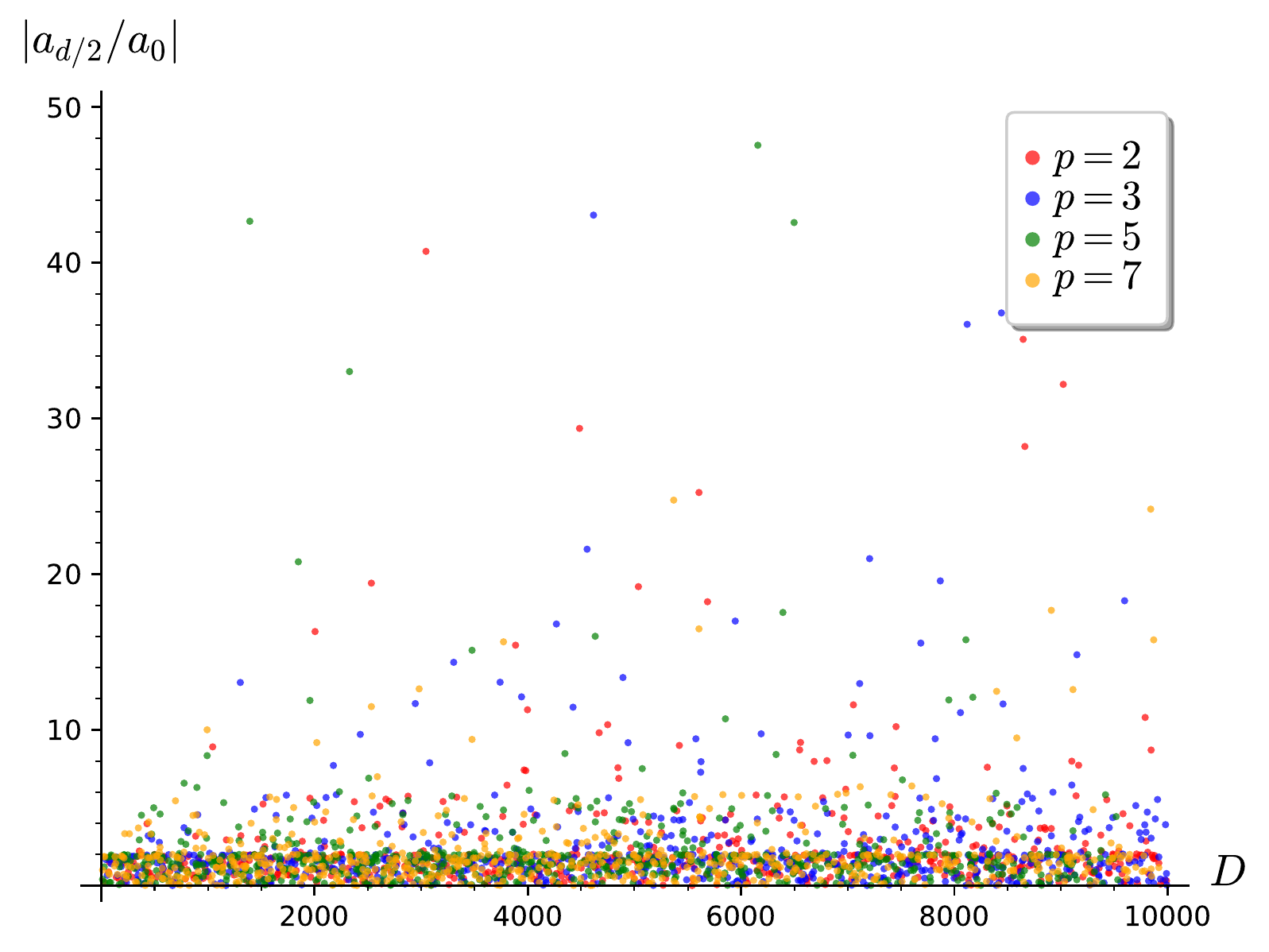}
  \caption{Normalised middle coefficients for various primes $p$.}
  \label{fig:middle-coeffs}
\end{figure}

\addcontentsline{toc}{section}{References}
\printbibliography
\end{document}